\documentclass[12pt, A4paper]{article}
\def\aa{\alpha}   \def\ba{\beta}    \def\da{\delta}   \def\Da{\Delta}
\def\ga{\gamma}      \def\ta{\theta}   
   
   \def\Sa{\Sigma}     
  \def\vn{\varepsilon}
    
\def\iy{\infty} \def\sbt{\subset}  \def\fc{\frac} \def\st{\sqrt}
  \def\pl{\partial}
\def\ts{\times}  
\def\na{\nabla} 
\def\bt{\bot}

\def\mb{\mathbb}  \def\ml{\mathcal}      
\def\mm{\mathrm}
\def\lt{\left}       \def\rt{\right}
\def\lae{\langle}    \def\rae{\rangle}
\def\ls{\limits}     \def\mx{\mbox}
       \def\ct{\cdot}      \def\lts{\ldots}
\def\ue{\underline}

\usepackage{amsmath, amssymb, amsthm, psfrag, graphicx}
\theoremstyle{plain}
\newtheorem{thm}{\textrm{Theorem}}
\newtheorem*{thm1}{\textrm{Theorem}}

\newtheorem{lemma}[thm]{\textrm{Lemma}}
\newtheorem{prop}[thm]{\textrm{Proposition}}

\theoremstyle{definition}
\newtheorem{defn}[thm]{\textrm{Definition}}

\newtheorem*{DP}{\textrm{($\star$) Dirichlet problem for the {\small SS-CMC} equation with symmetric boundary data}}

\theoremstyle{remark}

\newtheoremstyle{note}
     {3pt}
     {3pt}
     {}
     {}
     {\itshape}
     {}
     {.5em}
     {}
\theoremstyle{note}

\author{Kuo-Wei Lee}
\title{Dirichlet problem for the constant mean curvature equation
and CMC foliation in the extended Schwarzschild spacetime}
\begin{document}
\fontsize{12}{18pt plus.5pt minus.4pt}\selectfont
\maketitle
\begin{abstract}
We prove the existence and uniqueness of the Dirichlet problem for spacelike,
spherically symmetric, constant mean curvature equation with symmetric boundary data in the extended Schwarzschild spacetime.
As an application, we completely solve the {\small CMC} foliation conjecture which is posted by Malec and \'{O} Murchadha in \cite{MO}.
\end{abstract}

\section{Introduction}
Spacelike constant mean curvature ({\small CMC}) hypersurfaces in spacetimes have been considered as important objects in general relativity.
This is because constant mean curvature hypersurfaces are used in the analysis on Einstein constraint equations
\cite{L, CYY} and the gauge condition in the Cauchy problem of the Einstein equations \cite{AM, CYR}.
In addition, York suggested the concept of the {\small CMC} foliation and the
{\small CMC} time function on relativistic cosmology \cite{Y}, so {\small CMC} hypersurfaces could characterize
the global structure of cosmological spacetimes.

The Schwarzschild spacetime is the simplest model of a universe containing a star.
Its metric is a solution of the vacuum Einstein equation, and it is spherically symmetric and asymptotically flat.
The maximal analytic extension of the Schwarzschild spacetime was obtained by Kruskal in 1960 and nowadays it is called the Kruskal extension.

Some results of spacelike, spherically symmetric constant mean curvature hypersurfaces ({\small SS-CMC})
in the Schwarzschild spacetime and in the Kruskal extension can be found in Brill, Cavallo, and Isenberg's paper \cite{BCI},
and in Malec and \'{O} Murchadha's papers \cite{MO,MO2}.
In papers \cite{LL1, LL2},
the authors considered the initial value problem of the {\small SS-CMC} equation in the Schwarzschild spacetime.
These {\small SS-CMC} solutions are completely solved and characterized by two constants of integration.
Furthermore, the authors discussed the behaviors of {\small SS-CMC} hypersurfaces near the coordinate singularities,
and then gave the correspondences between {\small SS-CMC} hypersurfaces in the Schwarzschild spacetime and
{\small SS-CMC} hypersurfaces in the Kruskal extension.
Thus, the initial value problem of the {\small SS-CMC} equation in the Kruskal extension is solved as well.

In this paper, we consider the Dirichlet problem for the {\small SS-CMC} equation in the Kruskal extension with $T$-axisymmetric boundary data.
The result is comprehensively stated below:

\begin{thm1} \label{existence}
The Dirichlet problem for the {\small SS-CMC} equation and symmetric boundary data in the Kruskal extension is solvable and the solution is unique.
\end{thm1}

The precise setting of the {\small SS-CMC} Dirichlet problem is in section~\ref{SettingDP}.
For the existence part,
we use the shooting method since the boundary value problem can be reduced to the initial value problem,
which is solved in papers \cite{LL1, LL2}.
The uniqueness part is much difficult.
We introduce the Lorentizian distance function with respect to some spacelike hypersurface.
The maximum point of the Laplacian of the Lorentzian distance function restrict on another spacelike hypersurface
has a good estimate related to the mean curvatures of both spacelike hypersurfaces.
Thus, if the solution of the {\small SS-CMC} Dirichlet problem is not unique,
we can apply the Lorentzian distance function estimate to two of {\small SS-CMC} solutions and then get a contradiction.

As an application, we prove that the $T$-axisymmetric {\small SS-CMC} foliation (we use {\small TSS-CMC} foliation for short)
property conjectured by Malec and \'{O} Murchadha in \cite{MO} is true.
In paper \cite{MO}, they constructed a particular family of {\small TSS-CMC} hypersurfaces and claimed this family foliates the Kruskal extension.
In paper \cite{LL2}, the authors reformulated the {\small TSS-CMC} foliation conjecture and proved some partial results.
In this paper, we explain the existence of the Dirichlet problem is equivalent to the family of
{\small TSS-CMC} hypersurfaces cover the Kruskal extension,
and the uniqueness of the Dirichlet problem is equivalent to any two {\small TSS-CMC} hypersurfaces are disjoint,
and hence the {\small TSS-CMC} foliation conjecture is proved.

The organization of this paper is as follows.
We first give a brief summary of the Schwarzschild spacetime and Kruskal extension in section~\ref{SchwarzschildandKruskal},
and then state the main results of the initial value problem of the {\small SS-CMC} equation in the Kruskal extension in section~\ref{SSCMCinKruskal}.
The Lorentzian distance function and its properties are discussed in section 3.
These results are used to solve the {\small SS-CMC} Dirichlet problem in section 4.
In section 5, we will prove the {\small TSS-CMC} foliation conjecture.

The author would like to thank Yng-Ing Lee and Mao-Pei Tsui for their interests and discussions.
The author is supported by the NSC research grant 103-2115-M-002-013-MY3.

\section{Preliminary}
\subsection{The Schwarzschild spacetime and the Kruskal extension} \label{SchwarzschildandKruskal}
The Schwarzschild spacetime is a $4$-dimensional time-oriented Lorentzian manifold with metric
\begin{align*}
\mathrm{d}s^2=-\left(1-\frac{2M}r\right)\mathrm{d}t^2+\frac1{\left(1-\frac{2M}r\right)}\,\mathrm{d}r^2
+r^2\,\mathrm{d}\theta^2+r^2\sin^2\theta\,\mathrm{d}\phi^2.
\end{align*}
After coordinates change, the Schwarzschild metric can be written as
\begin{align}
\mathrm{d}s^2&=\frac{16M^2\mathrm{e}^{-\frac{r}{2M}}}{r}(-\mathrm{d}T^2+\mathrm{d}X^2)+r^2\,\mathrm{d}\theta^2+r^2\sin^2\theta\,\mathrm{d}\phi^2 \notag \\
&=\frac{16M^2\mathrm{e}^{-\frac{r}{2M}}}{r}\,\mathrm{d}U\mathrm{d}V+r^2\,\mathrm{d}\theta^2+r^2\sin^2\theta\,\mathrm{d}\phi^2, \label{SchMetric2}
\end{align}
where
\begin{align}
\left\{
\begin{array}{l}
\displaystyle(r-2M)\,\mathrm{e}^{\frac{r}{2M}}=X^2-T^2=VU\\
\displaystyle\frac{t}{2M}=\ln\left|\frac{X+T}{X-T}\right|=\ln\left|\frac{V}{U}\right|. \label{trans}
\end{array}
\right.
\end{align}
From (\ref{SchMetric2}), we know that $r=2M$ is only a coordinate singularity.
The Schwarzschild spacetime has a maximal analytic extension, called the Kruskal extension.
It is the union of regions {\tt I}, {\tt I$\!$I}, {\tt I'}, and {\tt I$\!$I'},
where regions {\tt I} and {\tt I$\!$I} correspond to the exterior and interior of one Schwarzschild spacetime, respectively,
and regions {\tt I'} and {\tt I$\!$I'} correspond to the exterior and interior of another Schwarzschild spacetime.
Figure~\ref{KruskalSimple} points out their correspondences and coordinate systems $(X,T)$ or $(U,V)$.

\begin{figure}[h]
\psfrag{A}{\tt I}
\psfrag{B}{\tt I$\!$I}
\psfrag{C}{\tt I$\!$I'}
\psfrag{D}{\tt I'}
\psfrag{E}{$U$}
\psfrag{F}{$V$}
\psfrag{X}{$X$}
\psfrag{T}{$T$}
\psfrag{r}{$r$}
\psfrag{t}{$t$}
\psfrag{M}{\tiny$2M$}
\psfrag{P}{\tiny $\partial_T$}
\centering
\includegraphics[height=50mm,width=71mm]{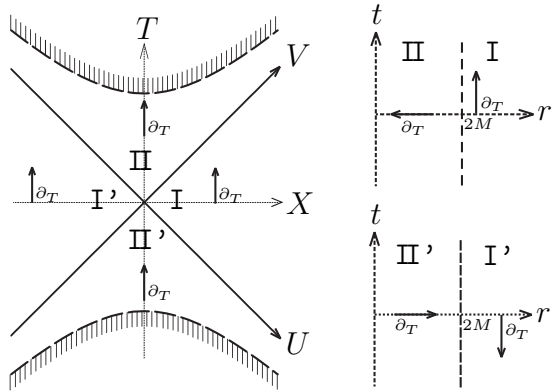}
\caption{The Kruskal extension of Schwarzschild spacetimes.} \label{KruskalSimple}
\end{figure}

Sometimes we will use another null coordinates $(u,v)$ by
\begin{align}
u=t-(r+2M\ln|r-2M|)\quad\mbox{and}\quad v=t+(r+2M\ln|r-2M|), \label{NullCoordinatesuv}
\end{align}
and relations between $(U,V)$ and $(u,v)$ are given by
\begin{align*}
\begin{array}{ccccc}
& \quad \mbox{region {\tt I}} & \quad\mbox{region {\tt I$\!$I}} & \quad \mbox{region {\tt I'}}
& \quad \mbox{region {\tt I$\!$I'}}\\
\hline
U & \quad \mathrm{e}^{-\frac{u}{4M}} & \quad -\mathrm{e}^{-\frac{u}{4M}}
& \quad -\mathrm{e}^{-\frac{u}{4M}} & \quad \mathrm{e}^{-\frac{u}{4M}} \\
V &\quad \mathrm{e}^{\frac{v}{4M}}  & \quad \mathrm{e}^{\frac{v}{4M}}
&\quad -\mathrm{e}^{\frac{v}{4M}}  & \quad-\mathrm{e}^{\frac{v}{4M}}. \\
\end{array}
\end{align*}

In this article, we will take $\partial_T$ as a future-directed timelike vector field in the Kruskal extension.
Note that $\partial_T$ in two Schwarzschild spacetimes has different directions and it is pointed out in Figure~\ref{KruskalSimple}.
Once $\pl_T$ is chosen, for a spacelike hypersurface $\Sa$,
we will choose the normal vector $\vec{n}$ of $\Sa$ as future-directed in the Kruskal extension,
and the mean curvature $H$ of $\Sa$ is defined by $H=\fc13g^{ij}\lae\na_{e_i}\vec{n},e_j\rae$,
where $\{e_i\}_{i=1}^3$ is a basis on $\Sa$.

\subsection{{\small SS-CMC} hypersurfaces in the Kruskal extension} \label{SSCMCinKruskal}
In this subsection,
we will summarize the results in papers \cite{LL1} and \cite{LL2} about spherically symmetric, spacelike constant mean curvature ({\small SS-CMC})
hypersurfaces in the Schwarzschild spacetime and Kruskal extension.
These formulae and arguments will be used to deal with the {\small SS-CMC} Dirichlet problem in section~\ref{section4}.
We refer to papers \cite{LL1, LL2} for more details.

In papers \cite{LL1, LL2},
we considered the initial value problem of the {\small SS-CMC} equation in the Schwarzschild spacetime and Kruskal extension.
We first studied {\small SS-CMC} hypersurfaces in the standard Schwarzschild coordinates.
In most of cases, an {\small SS-CMC} hypersurface $\Sa$ can be locally written as a graph of
$(t=f(r),r,\ta,\phi)$ in either exterior or interior of the Schwarzschild spacetime.
Direct computation gives the {\small SS-CMC} equation, which is a second-order ordinary differential equation:
\begin{align}
f''+\left(\left(\frac1{h}-(f')^2h\right)\left(\frac{2h}{r}+\frac{h'}2\right)+\frac{h'}{h}\right)f'\pm 3H\left(\frac1{h}-(f')^2h\right)^{\frac32}=0,
\label{SSCMCeqn}
\end{align}
where $h(r)=1-\fc{2M}{r}$, $H\in\mb{R}$ is the constant mean curvature,
and the spacelike condition is equivalent to $\frac1{h}-(f')^2h>0$.
The choice of $\pm$ signs in the equation (\ref{SSCMCeqn}) depends on different regions
that cause different directions of the normal vector and different pieces of an {\small SS-CMC} hypersurface.
In each region and where $f(r)$ is defined, the solution $f(r)$ is uniquely determined by two constants of integration, denoted by $c$ and $\bar{c}$,
where $c$ controls the slope of the function of $f(r)$ and $\bar{c}$ represents the $t$-direction translation.
Thus, if we require that an {\small SS-CMC} hypersurface $\Sa_{H,c,\bar{c}}$ passes through a point $(t_0,r_0)$, then $\bar{c}$ is determined.

All solutions of the {\small SS-CMC} equation can be expressed in the integration form,
which are completely discussed in \cite{LL1, LL2}.
Because of symmetry, only parts of expressions of $\Sa_{H,c,\bar{c}}$ are needed in this paper,
and we summarize these formulae below.
\begin{itemize}
\item[(A)]
For $\Sa_{H,c,\bar{c}}$ in the Schwarzschild exterior (maps to region {\tt I}), then
\begin{align*}
f(r;H,c,\bar{c})=\int_{r_1}^r\frac{l(x;H,c)}{h(x)\st{1+l^2(x;H,c)}}\,\mathrm{d}x+\bar{c},
\end{align*}
where $l(r;H,c)=\frac{1}{\sqrt{h(r)}}\left(Hr+\frac{c}{r^2}\right)$, and $r_1\in(2M,\iy)$ is a fixed number.\\[2mm]
\ue{If $c<-8M^3H$}, then $f(r)\to\iy$ as $r\to 2M^+$ with asymptotic behavior $f'(r)=\fc1{h(r)}+\bar{f}'(r)$, where
\begin{align}
\bar{f}'(r)=\fc{r^4}{(Hr^3+c)^2+r^3(r-2M)-(Hr^3+c)\st{(Hr^3+c)^2+r^3(r-2M)}}. \label{barfeqn}
\end{align}
\ue{If $c>-8M^3H$}, then $f(r)\to-\iy$ as $r\to 2M^+$ with asymptotic behavior $f'(r)=-\fc1{h(r)}+\tilde{f}'(r)$, where
\begin{align}
\tilde{f}'(r)=\fc{-r^4}{(Hr^3+c)^2+r^3(r-2M)+(Hr^3+c)\st{(Hr^3+c)^2+r^3(r-2M)}}. \label{tildefeqn}
\end{align}
\ue{If $c=-8M^3H$}, then $f(r)$ tends to some finite value as $r\to 2M^+$.
\item[(B)]
For $\Sa_{H,c,\bar{c}}$ in the Schwarzschild interior (region {\tt I\!I}), then
\begin{align*}
f(r;H,c,\bar{c})
=\int_{r_2}^r\fc{l(x;H,c)}{\mp h(x)\st{l^2(x;H,c)-1}}\,\mathrm{d}x+\bar{c},
\end{align*}
where $l(r;H,c)=\frac1{\sqrt{-h(r)}}\left(-Hr-\frac{c}{r^2}\right)>1$, and $r_2$ is in the domain of $f(r)$.\\[2mm]
If $r\to 2M^-$ and $f'(r)>0$, then $f(r)\to\iy$ with asymptotic behavior $f'(r)=-\fc1{h(r)}+\bar{f}'(r)$, where $\bar{f}'(r)$ is (\ref{barfeqn}).
\item[(C)]
For $\Sa_{H,c,\bar{c}}$ in the Schwarzschild interior (region {\tt I\!I'}), then
\begin{align*}
f(r;H,c,\bar{c})
=\int_{r_4}^r\frac{l(x;H,c)}{\pm h(x)\st{l^2(x;H,c)-1}}\,\mathrm{d}x+\bar{c},
\end{align*}
where $l(r;H,c)=\frac1{\sqrt{-h(r)}}\left(Hr+\frac{c}{r^2}\right)>1$, and $r_4$ is in the domain of $f(r)$.\\[2mm]
If $r\to 2M^-$ and $f'(r)<0$, then $f(r)\to-\iy$ with asymptotic behavior $f'(r)=\fc1{h(r)}+\tilde{f}'(r)$, where $\bar{f}'(r)$ is (\ref{tildefeqn}).
\item[(D)] \ue{If $c<-8M^3H$},
we can find $\bar{c}$ such that hypersurfaces in region {\tt I} and {\tt I\!I} can smoothly glue together at $U=0$ ($r=2M$ and $t=\iy$).\\
\ue{If $c>-8M^3H$},
we can find $\bar{c}$ such that hypersurfaces in region {\tt I} and {\tt I\!I'} can smoothly glue together at $V=0$ ($r=2M$ and $t=-\iy$).\\
\ue{If $c=-8M^3H$},
we can find $\bar{c}$ such that hypersurfaces in region {\tt I} and {\tt I'} can smoothly glue together at $(U,V)=(0)$ ($r=2M$ and $t$ is finite).
\end{itemize}

Only constant slices $r=r_0, r_0\in(0,2M)$ are {\small SS-CMC} hypersurfaces which can {\it not} be written as a graph of $t=f(r)$.
These hypersurfaces are called {\it cylindrical hypersurfaces} and they have constant mean curvature
\begin{align*}
H(r_0)=\fc{2r_0-3M}{3\st{r_0^3(2M-r_0)}} \mx{ in region } {\tt I\!I}, \mx{ or } H(r_0)=\fc{3M-2r_0}{3\st{r_0^3(2M-r_0)}} \mx{ in region } {\tt I\!I'}.
\end{align*}

The behavior of an {\small SS-CMC} hypersurface $\Sa_{H,c,\bar{c}}$ in the Kruskal extension highly depends on the constant of integration $c$.
Precisely, fixed $H\in\mb{R}$, define two functions $k_H(r)$ and $\tilde{k}_H(r)$ on $(0,2M]$ by
\begin{align*}
k_{H}(r)=-Hr^3-r^{\frac32}(2M-r)^{\frac12}\quad\mx{and}\quad
\tilde{k}_{H}(r)=-Hr^3+r^{\frac32}(2M-r)^{\frac12}.
\end{align*}
These two functions come from the condition $l>1$ in the Schwarzschild interior and their graphs are plotted in Figure~\ref{DomainofCMC} (a)\footnote{
All figures in this paper are presented the case $H>0$. Cases $H=0$ or $H\leq 0$ are similarly discussed.}.

Denote $(r_H,c_H)$ by the minimum point of $k_H(r)$ and $(R_H,C_H)$ by the maximum point of $\tilde{k}_H(r)$.
We use $k_H^+(r)$, $k_H^-(r)$, $\tilde{k}_H^+(r)$, and $\tilde{k}_H^-(r)$ to represent their increasing or decreasing part.
Notice that both $r=r_H$ and $r=R_H$ are cylindrical {\small SS-CMC} hypersurfaces with constant mean curvature
$H$ in region {\tt I\!I} and {\tt I\!I'}, respectively.
They correspond to hyperbolas in the Kruskal extension, and thus they divide the Kruskal extension into three regions,
called the top region, the middle region, and the bottom region. See Figure~\ref{DomainofCMC} (b).

\begin{figure}[h]
\hspace*{5mm}
\psfrag{r}{$r$}
\psfrag{t}{$t$}
\psfrag{k}{$k_H(r)$}
\psfrag{K}{$\tilde{k}_H(r)$}
\psfrag{R}{\footnotesize$R_H$}
\psfrag{S}{\footnotesize$r_H$}
\psfrag{M}{\footnotesize$2M$}
\psfrag{H}{$-8M^3H$}
\psfrag{C}{$C_H$}
\psfrag{c}{$c_H$}
\psfrag{l}{\hspace*{-1mm}$L(r)=c$}
\psfrag{P}{$r=0$}
\psfrag{Q}{$r=r_H$}
\psfrag{U}{$r=R_H$}
\psfrag{V}{$r=0$}
\psfrag{A}{\footnotesize top region}
\psfrag{B}{\footnotesize middle region}
\psfrag{N}{\footnotesize bottom region}
\psfrag{Y}{(b)}
\psfrag{Z}{(a)}
\includegraphics[height=55mm,width=121mm]{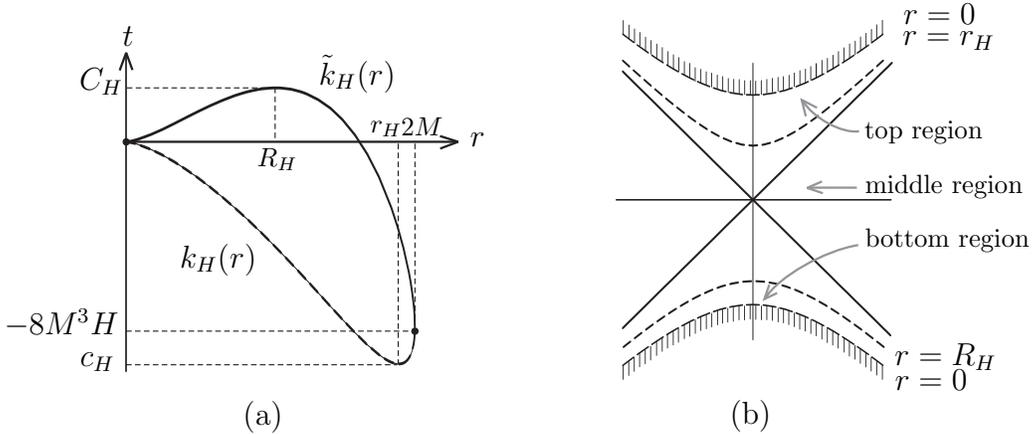}
\caption{(a) Graphs of $k_H(r)$ and $\tilde{k}_H(r)$.
(b) {\small SS-CMC} hypersurfaces $r=r_H$ and $r=R_H$,
and they divide the Kruskal extension into three regions.} \label{DomainofCMC}
\end{figure}

Now we summarize results of the initial value problem of the {\small SS-CMC} equation in the Kruskal extension.
Given $H\in\mb{R}$, for $(T_0,X_0)$ in the middle region as Figure~\ref{IVPcaseI}, and for $c\in\mb{R}$,
we can find an {\small SS-CMC} hypersurface $\Sa_{H,c}$ passing through $(T_0,X_0)$.
There are three values $c=c_H, c=-8M^3H$, and $c=C_H$ such that the value $c$ in different interval
determines different behavior of the {\small SS-CMC} hypersurface $\Sa_{H,c}$, which is also illustrated in Figure~\ref{IVPcaseI}.
Remark that for $c\in(c_H,C_H)$, the solution of $c=k_H^+(r)$ or $c=\tilde{k}_H^-(r)$, which is denoted by $r=r_{H,c}$,
is called ``throat" of the {\small SS-CMC} hypersurface $\Sa_{H,c}$.
The throat indicates that $\Sa_{H,c}$ is not defined in $r\in(0,r_{H,c})$ in the Schwarzschild interior.
Notice that each $\Sa_{H,c}$ is diffeomorphic to $I\ts\mb{S}^2$, where $I\sbt\mb{R}$,
so the throat $r=r_{H,c}$ is the smallest radius of the Schwarzschild coordinates sphere in $\Sa_{H,c}$.
Furthermore, every {\small SS-CMC} hypersurface $\Sa_{H,c,\bar{c}}, c\in(c_H,C_H)$ is symmetric with respect to its throat in the Schwarzschild coordinates.
See the discussion in \cite[Proposition 2.6]{LL1}.

\begin{figure}[h]
\psfrag{A}{\tiny$(T_0,X_0)$}
\psfrag{B}{\tiny$(T_0,-X_0)$}
\psfrag{a}{\footnotesize$c<c_H$}
\psfrag{b}{\footnotesize$c=c_H$}
\psfrag{c}{\hspace*{-19mm}\footnotesize$c_H<c<-8M^3H$}
\psfrag{d}{\hspace*{-10mm}\footnotesize$c=-8M^3H$}
\psfrag{e}{\hspace*{-20mm}\footnotesize$-8M^3H<c<C_H$}
\psfrag{f}{\hspace*{-2mm}\footnotesize$c=C_H$}
\psfrag{g}{\hspace*{-2mm}\footnotesize$c>C_H$}
\psfrag{r}{$r$}
\psfrag{t}{$t$}
\psfrag{k}{$k_H^+(r)$}
\psfrag{K}{$\tilde{k}_H^-(r)$}
\psfrag{R}{\footnotesize$R_H$}
\psfrag{S}{\footnotesize$r_H$}
\psfrag{M}{\footnotesize$2M$}
\psfrag{H}{\footnotesize$-8M^3H$}
\psfrag{C}{\footnotesize$C_H$}
\psfrag{D}{\footnotesize$c_H$}
\psfrag{l}{\hspace*{-1mm}$L(r)=c$}
\includegraphics[height=50mm,width=140mm]{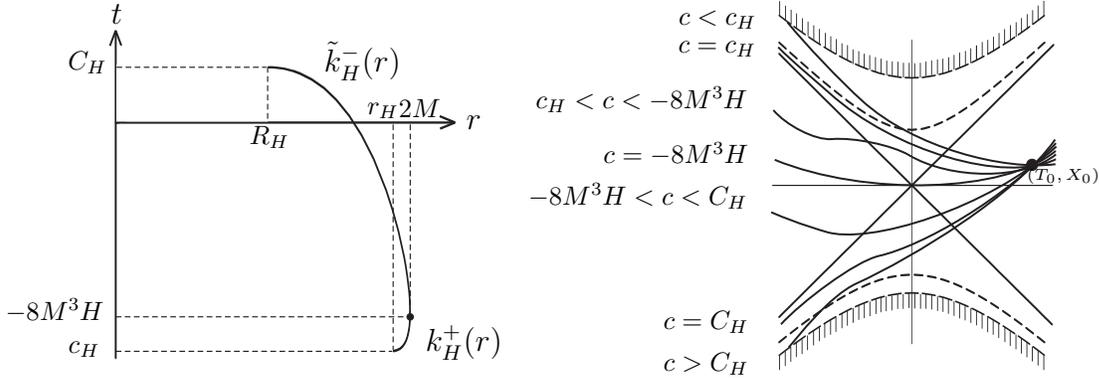}
\caption{{\small SS-CMC} initial value problem when $(T_0, X_0)$ is in the middle region.} \label{IVPcaseI}
\end{figure}

For $(T_0,X_0)$ in the top region (or the bottom region), as Figure~\ref{IVPcaseII},
denote $(t_0,r_0)$ by its Schwarzschild coordinates,
then for $c\leq k_H^-(r_0)$ (or $c\geq\tilde{k}_H^+(r_0)$), we can find an {\small SS-CMC} hypersurface $\Sa_{H,c}$ with $f'(r_0)<0$ (or $f'(r_0)>0$)
in the Schwarzschild interior passing through $(T_0,X_0)$.
There is also an {\small SS-CMC} hypersurface $\Sa_{H,c}$ with $f'(r_0)>0$ (or $f'(r_0)<0$) in the Schwarzschild interior passing through $(T_0,X_0)$,
but it is not illustrated in Figure~\ref{IVPcaseII}.

\begin{figure}[h]
\centering
\psfrag{A}{\tiny$(T_0,X_0)$}
\psfrag{B}{\tiny$(T_0,-X_0)$}
\psfrag{c}{\footnotesize$c<c_H$}
\psfrag{b}{\footnotesize$c=c_H$}
\psfrag{a}{\hspace*{-7mm}\footnotesize$c_H<c<0$}
\psfrag{g}{\hspace*{-8.5mm}\footnotesize$0<c<C_H$}
\psfrag{f}{\hspace*{-2mm}\footnotesize$c=C_H$}
\psfrag{e}{\hspace*{-2mm}\footnotesize$c>C_H$}
\psfrag{r}{$r$}
\psfrag{t}{$t$}
\psfrag{k}{$k_H^-(r)$}
\psfrag{K}{$\tilde{k}_H^+(r)$}
\psfrag{R}{\footnotesize$R_H$}
\psfrag{S}{\footnotesize$r_H$}
\psfrag{M}{\footnotesize$2M$}
\psfrag{H}{\footnotesize$-8M^3H$}
\psfrag{C}{\footnotesize$C_H$}
\psfrag{D}{\footnotesize$c_H$}
\psfrag{l}{\hspace*{-1mm}$L(r)=c$}
\includegraphics[height=50mm,width=134mm]{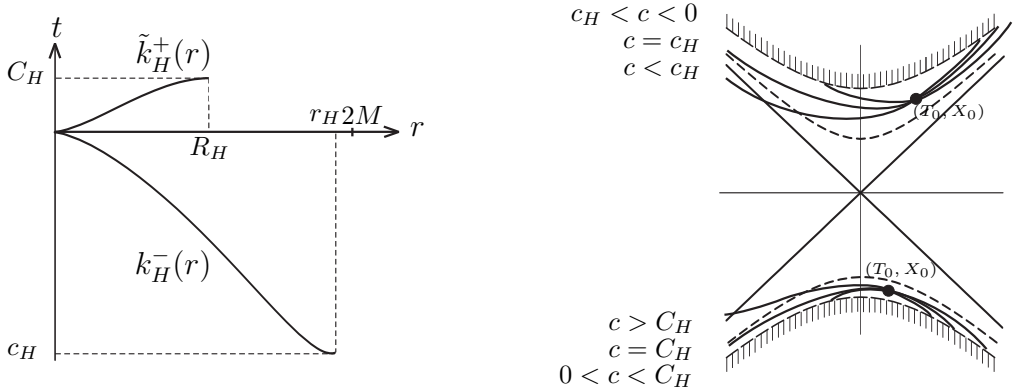}
\caption{{\small SS-CMC} initial value problem when $(T_0, X_0)$ is in the top or bottom region.} \label{IVPcaseII}
\end{figure}

\newpage
\section{The Lorentzian distance function}
We will introduce the Lorentzian distance function from a fixed achronal spacelike hypersurface in this section,
and then derive a formula of the Laplacian of the Lorentzian distance function
for further analysis on the Dirichlet problem of the {\small SS-CMC} equation in section~\ref{section4}.
Here we focus on the spacetime case.
Remark that results in this section hold in general $n$-dimensional Lorentzian manifold.
Most of results in this section can also be found in papers \cite{AHP, EGK} and in their references.

Let $N^3$ be an achronal spacelike hypersurface in a spacetime $(L^4,\lae\ct,\ct\rae)$.
Given $p, q\in L$, define the {\it Lorentzian distance function} $d(p,q):L\ts L\to[0,\iy]$ as follows:
\begin{itemize}
\item[(a)] If $q\in J^+(p)$, then $d(p,q)$ is the supremum of the Lorentzian lengths of all the future-directed causal curve from $p$ to $q$.
\item[(b)] If $q\not\in J^+(p)$, then $d(p,q)=0$.
\end{itemize}
Remark that the Lorentzian distance function is {\it not} symmetric,
that is, $d(p,q)\neq d(q,p)$.
For $q\in L$, we can further define the {\it Lorentzian distance function $d_N:L\to[0,\iy]$ with respect to $N$} as $d_N(q)=\sup\ls_{p\in N}d(p,q)$.

In general, the Lorentzian distance function $d_N$ is {\it not} smooth for every point in $L$,
but it is smooth in a sufficiently near chronological future of $N$.
To state this result, let $\nu$ be the future-directed unit timelike normal vector field of $N$.
We define a function $s_N:N\to[0,\iy]$ by $s_N(p)=\sup\{t\geq 0: d_N(\ga(t))=t\}$,
where $\ga(t)$ is the future inextensible geodesic such that $\ga(0)=p$ and $\ga'(0)=\nu$.
Denote $\tilde{\ml{I}}^+(N)=\{t\nu:\mx{ for all } p\in N \mx{ and } 0<t<s_N(p)\}$,
and consider the set
\begin{align*}
\ml{I}^+(N)=\exp_N(\mm{int}(\tilde{\ml{I}}^+(N)))\sbt I^+(N),
\end{align*}
where $\exp_N$ is the exponential map with respect to $N$.

\begin{lemma}{\rm \cite[Lemma 5.1]{AHP}, \cite[Proposition 3.6]{EGK}}
If $\ml{I}^+(N)\neq\emptyset$, then $d_N$ is smooth on $\ml{I}^+(N)$
and its gradient $\overline{\na}d_N$ is a past-directed unit
timelike vector field on $\ml{I}^+(N)$. That is,
$\overline{\na}d_N=-\nu$.
\end{lemma}
The analysis on the Lorentzian distance function is related to the
Jacobi field, and we summarize the results here. Denote
$\overline{\na}$ the connection of the spacetime $L$ compatible with
the metric $\lae\ct,\ct\rae$, and $\na$ the connection of $N$
induced from $L$. Define $A_N$ as the {\it shape operator} of $N$
with respect to $L$ by $A_N(X)=-\na_X\nu$ for future-directed unit
timelike normal vector field $\nu$ of $N$ and for tangent vector
field $X$ of $N$. Given $p\in N$, let $\ga(t)$ be the
future-directed unit timelike geodesic with $\ga(0)=p$ and
$\ga'(0)=\nu$. We say a vector field $J(t)$ on the geodesic $\ga(t)$
is a {\it normal $N$-Jacobi vector field} if
$J''(t)+R(J(t),\ga'(t),\ga'(t))=0$, $J(0)\in T_pN$, and
$J'(0)=-A_N(J(0))$.

When we write a normal $N$-Jacobi field $J(t)$ in term of a parallel orthonormal frame,
by the existence and uniqueness of the ordinary differential system, for $v\in T_pN$,
there exists a unique $N$-Jacobi field $J(t)$ along $\ga(t)$ satisfies $J(0)=v$,
so the set of all normal $N$-Jacobi fields $J(t)$ along $\ga(t)$ is of dimension $n$.

The point $q\in L$ is called a {\it focal point} of $N$ if there exists a geodesic $\ga:[0,l]\to L$ with
$\ga(0)=p\in N, \ga'(0)\in(T_pN)^\bt, \ga(l)=q\in L$, and a nonzero normal $N$-Jacobi field $J(t)$ along $\ga(t)$ such that $J(l)=0$.

Similar to the Riemannian case, we can compute the Hessian of the
Lorentzian distance function $\overline{\mm{Hess}}(d_N)$ in terms of
the normal $N$-Jacobi vector field $J$ and the curvature tensor $R$
of the spacetime $L$.
\begin{prop}{\rm\cite[Proposition 3.7]{EGK}}
Let $N$ be an achronal spacelike hypersurface in a spacetime $L$. If
$\ml{I}^+(N)\neq\emptyset, q\in\ml{I}^+(N)$ and $Y\in T_qL$ is a
nonzero vector orthogonal to $\overline{\na}d_N(q)$, then
\begin{align*}
\overline{\mm{Hess}}(d_N(q))(Y,Y)
=&-\int_0^l(\lae J'(t),J'(t)\rae-\lae R(J(t),\ga'(t))\ga'(t),J(t)\rae)\,\mm{d}t\\
&+\lae A_N(J(0)),J(0)\rae,
\end{align*}
where $\ga(t):[0,l]\to\ml{I}^+(N)$, $\ga(0)=p, \ga'(0)\in(T_pN)^\bt$ and $\ga(l)=q$ is the future-directed unit timelike geodesic,
$J(t)$ is the unique normal $N$-Jacobi field along $\ga(t)$ with $J(l)=Y$, and $J'(0)=-A_N(J(0))$.
\end{prop}

\begin{proof}
Since $\ga(t)$ is realized by the Lorentzian distance function, we
have $\ga'(t)=-\overline{\na} d_N(\ga(t))$ for all $t\in(0,l]$. If
$J(t)$ is the unique normal $N$-Jacobi field along $\ga(t)$, then we
have
\begin{align*}
&\quad\ \overline{\mm{Hess}}(d_N)(Y,Y)
=\overline{\mm{Hess}}(d_N)(J(l),J(l)) \\
&=\lae J(l),\overline{\na}_{J(l)}\overline{\na} d_N\rae
=-\lae J(l),\overline{\na}_{\ga'(l)}J(l)\rae=-\lae J(l),J'(l)\rae \\
&=-\int_0^l\fc{\mm{d}}{\mm{d}t}\lae J(t),J'(t)\rae\,\mm{d}t-\lae J(0),J'(0)\rae \\
&=-\int_0^l\lt(\lae J'(t),J'(t)\rae-\lae R(J(t),\ga'(t))\ga'(t),J(t)\rae\rt)\mm{d}t+\lae A_N(J(0)),J(0)\rae.
\end{align*}
\end{proof}
From the formula of $\overline{\mm{Hess}}(d_N)$, it is natural to
define the {\it index form} of the geodesic $\ga(t)$ with respect to
$N$:
\begin{align*}
I_N(V,W)=-\int_0^l(\lae V',W'\rae-\lae R(V,\ga')\ga',W\rae)\,\mm{d}t+\lae A_N(V(0)),W(0)\rae.
\end{align*}
where $V$ and $W$ are vector fields along and orthogonal to $\ga$.
\newpage
The following theorem shows that the normal $N$-Jacobi vector field maximizes the index form $I_N$.
\begin{thm}{\rm\cite[Theorem 5.4]{AHP}}
Let $N$ be an achronal spacelike hypersurface in a spacetime $L$.
Given $p\in N$, let $\ga(t):[0,l]\to\ml{I}^+(N)$, $\ga(0)=p$ and $\ga'(0)\in T_pN$ be the future-directed unit timelike geodesic.
Suppose there are no focal points to $N$ along $\ga$.
Let $J$ be a normal $N$-Jacobi field along $\ga$.
Then for every vector field $X$ along and orthogonal to $\ga$ such that $X(l)=J(l)$,
it holds that
\begin{align*}
I_N(J,J)\geq I_N(X,X),
\end{align*}
with equality if and only if $J=X$.
\end{thm}

\begin{proof}
Let $\{e_i(t)\}_{i=1}^3$ be a parallel orthonormal frame along $\ga$ and orthogonal to $\ga'(t)$,
then we can write
\begin{align*}
J(t)=\sum_{i=1}^3J_i(t)e_i(t),\ \mx{ and }\ X(t)=\sum_{i=1}^3X_i(t)J_i(t)e_i(t),\ X_i(l)=1.
\end{align*}
It implies $X'(t)=\sum\ls_{i=1}^3 X_i'(t)J_i(t)e_{i}(t)+X_i(t)J_i'(t)e_i(t)=A(t)+B(t)$, and
\begin{align*}
I_N(X,X)
=-\int_0^l\lae A,A\rae+2\lae A,B\rae+\lae B,B\rae-\lae R(X,\ga')\ga',X\rae\,\mm{d}t+\lae A_N(X(0)),X(0)\rae.
\end{align*}
Since
\begin{align*}
\lae X(t),B(t)\rae'
&=\lae X'(t),B(t)\rae+\lae X(t),B'(t)\rae \\
&=\lae A+B,B\rae+\lt\lae \sum_{i=1}^3X_iJ_ie_i,\sum_{j=1}^3X_j'J_j'e_j+X_jJ_j''e_j\rt\rae \\
&=2\lae A,B\rae+\lae B,B\rae+\sum_{i,j=1}^3\lae X_iJ_ie_i, X_jJ_j''e_j\rae,
\end{align*}
we have
\begin{align*}
I_N(X,X)
&=-\int_0^l\lae A,A\rae+\lae X(t),B(t)\rae'-\sum_{j=1}^3\lae X, X_jJ_j''e_j\rae-\lae R(X,\ga')\ga',X\rae\,\mm{d}t\\
&\quad\,+\lae A_N(X(0)),X(0)\rae \\
&=-\int_0^l\lae A,A\rae\,\mm{d}t-\lae X(l),B(l)\rae+\lae X(0),B(0)\rae+\lae A_N(X(0)),X(0)\rae \\
&=-\int_0^l\lae A,A\rae\,\mm{d}t+I_N(J,J)+\lae A_N(X(0))+B(0),X(0)\rae\leq I_N(J,J).
\end{align*}
\end{proof}
Next, we will get the estimate an Laplacian of the Lorentzian distance function,
where the spacetime satisfies the timelike convergence condition.
For $q\in L$, we say $p\in N$ is the {\it orthogonal projection of $q$ on $N$}
if there is a timelike geodesic $\ga:[0,l]\to L$ such that $\ga(0)=p, \ga(l)=q$ and $\ga'(0)\in(T_pN)^\bt$.

\begin{lemma}{\rm\cite[Lemma 5.7]{AHP}} \label{LemmaLaplacian}
Let $L$ be a spacetime such that $Ric(Z,Z)\geq 0$ for every unit timelike vector $Z$,
and let $N$ be an achronal spacelike hypersurface such that $\ml{I}^+(N)\neq\emptyset$ and let $q\in\ml{I}^+(N)$.
Then
\begin{align*}
\overline{\Da}d_N(q)\geq-3H_N(p),
\end{align*}
where $\overline{\Da}$ stands for the Lorentzian Laplacian operator
on $L$, $H_N$ is the mean curvature of the hypersurface $N$ with
respect to $\nu$, and $p$ is the orthogonal projection of $q$ on
$N$.
\end{lemma}

\begin{proof}
Let $\ga:[0,l]\to L$ be the normal future-directed unit timelike
geodesic with $\ga(0)=p$ and $\ga'(0)\in(T_pN)^\bot$. Let
$\{e_i\}_{i=1}^3$ be orthonormal vectors in $T_qL$ orthogonal to
$\ga'(l)=-\overline{\na}d_N(q)$ and $J_i(t)$ be a normal $N$-Jacobi
vector field along $\ga$ with $J_i(l)=e_i$ and $J_i'(0)=-A_NJ_i(0)$.
For every $\{X_{i}(t)\}_{i=1}^{4}$ orthonormal frame of parallel
vector fields along $\ga$ such that $X_i(l)=e_i$ for $i=1,\lts,3$
and $X_{4}(t)=\ga'(t)$, we have
\begin{align*}
\overline{\Da}d_N(q)
&=\sum_{i=1}^3\overline{\mm{Hess}}(d_N(q))(e_i,e_i)
=\sum_{i=1}^3\overline{\mm{Hess}}(d_N(q))(J_i(l),J_i(l)) \\
&=\sum_{i=1}^3I_N(J_i(l),J_i(l))
\geq\sum_{i=1}^3I_N(X_i(l),X_i(l)) \\
&=-\int_0^l\lt(\sum_{i=1}^3\lae X_i',X_i'\rae-\sum_{i=1}^3\lae R(X_i,\ga')\ga',X_i\rae\rt)\mm{d}t+\sum_{i=1}^3\lae A_NX_i(0),X_i(0)\rae \\
&=\int_0^l\mm{Ric}(\ga'(t),\ga'(t))\,\mm{d}t-3H_N(p)\geq-3H_N(p).
\end{align*}
\end{proof}

If we restrict the Lorentzian distance function $d_N:L\to[0,+\iy]$
on a spacelike hypersurface $\Sa$, which denotes
$d_N|_{\Sa}:\Sa\to[0,+\iy]$, we will find relations between
$\overline{\mm{Hess}}(d_N)$ and $\mm{Hess}(d_N|_{\Sa})$. Since
$\overline{\na}d_N=(\overline{\na}d_N)^\top+(\overline{\na}d_N)^\bot=\na(d_N|_{\Sa})-\lae\overline{\na}{d_N},\nu\rae\nu$
along $\Sa$, where $\na(d_N|_{\Sa})$ is the gradient of $d_N|_{\Sa}$
on $\Sa$, we have
\begin{align*}
\lae\overline{\na}d_N,\overline{\na}d_N\rae
&=\lae \na (d_N|_{\Sa})-\lae\overline{\na}{d_N},\nu\rae\nu,\na(d_N|_{\Sa})-\lae\overline{\na}{d_N},\nu\rae\nu\rae \\
&=|\na(d_N|_{\Sa})|^2-|\lae\overline{\na}{d_N},\nu\rae|^2
=-1.
\end{align*}
Since $\lae\overline{\na}d_N,\nu\rae>0$, we have
$\lae\overline{\na}d_N,\nu\rae=\st{1+|\na(d_N|_{\Sa})|^2}$ and hence
$\overline{\na}d_N=\na(d_N|_{\Sa})-\st{1+|\na(d_N|_{\Sa})|^2}\,\nu$.
Next, for any vector field $X\in T\Sa$, we have
\begin{align*}
\overline{\na}_X\overline{\na}d_N
&=\lt(\overline{\na}_X\overline{\na}d_N\rt)^\top+\lt(\overline{\na}_X\overline{\na}d_N\rt)^\bot \\
&=\lt(\overline{\na}_X\lt(\na(d_N|_{\Sa})-\st{1+|\na(d_N|_{\Sa})|^2}\,\nu\rt)\rt)^\top\\
&\quad+\lt(\overline{\na}_X\lt(\na(d_N|_{\Sa})-\st{1+|\na(d_N|_{\Sa})|^2}\,\nu\rt)\rt)^\bot \\
&=\na_X\na(d_N|_{\Sa})+\st{1+|\na(d_N|_{\Sa})|^2}\,A_{\Sa}(X)\\
&\quad-\lae A_{\Sa}(X),\na(d_N|_{\Sa})\rae\nu-X\lt(\st{1+|\na(d_N|_{\Sa})|^2}\rt)\nu.
\end{align*}
Thus,
\begin{align*}
\mm{Hess}(d_N|_{\Sa})(X,X)
&=\lae \na_X\na(d_N|_{\Sa}),X\rae \\
&=\overline{\mm{Hess}}(d_N)(X,X)-\st{1+|\na(d_N|_{\Sa})|^2}\lae A_{\Sa}(X),X\rae.
\end{align*}
After taking trace, we get for all $q\in\Sa$,
\begin{align}
\Da_{\Sa}(d_N|_{\Sa})(q)=\overline{\Da}d_N(q)+\overline{\mm{Hess}}(d_N(q))(\nu,\nu)+3H_{\Sa}(q)\st{1+|\na(d_N|_{\Sa})(q)|^2}. \label{Laplacian1}
\end{align}
Remark the the above discussion can also be found in \cite[page 5091]{AHP}.
The result of Lemma~\ref{LemmaLaplacian} and the equation (\ref{Laplacian1}) will imply an important inequality of $\Da_{\Sa}(d_N|_{\Sa})$.
\begin{prop}{\rm\cite[Proposition 5.9]{AHP}} \label{PropLaplacian}
Let $L$ be a spacetime such that $Ric(Z,Z)\geq 0$ for every unit timelike vector $Z$,
and let $N$ be an achronal spacelike hypersurface with $\ml{I}^+(N)\neq\emptyset$.
Suppose that a spacelike hypersurface $\Sa$ satisfies $\Sa\sbt\ml{I}^+(N)$.
Denote $d_N|_{\Sa}$ by the Lorentzian distance function with respect to $N$ on $\Sa$. Then
\begin{align}
\Da_{\Sa}(d_N|_{\Sa})(q)\geq\overline{\mm{Hess}}(d_N(q))(\nu,\nu)+3H_{\Sa}(q)\st{1+|\na(d_N|_{\Sa})(q)|^2}-3H_N(p), \label{EstimateofLaplace}
\end{align}
where $\nu$ and $H_{\Sa}$ are the future-directed unit timelike vector field and the mean curvature of $\Sa$, respectively,
$H_N$ is the future mean curvature of $N$ along the orthogonal projection of $\Sa$ on $N$,
and $p$ is the orthogonal projection of $q$ on $N$.
\end{prop}

\begin{proof}
Lemma~\ref{LemmaLaplacian} and equation (\ref{Laplacian1}) imply
\begin{align*}
\overline{\Da}d_N(q)
&=\Da_{\Sa}(d_N|_{\Sa})(q)-\overline{\mm{Hess}}(d_N(q))(\nu,\nu)-3H_{\Sa}(q)\st{1+|\na (d_N|_{\Sa})(q)|^2} \\
&\geq-3H_N(p),
\end{align*}
so we have
\begin{align*}
\Da_{\Sa}(d_N|_{\Sa})(q)\geq\overline{\mm{Hess}}(d_N(q))(\nu,\nu)+3H_{\Sa}(q)\st{1+|\na (d_N|_{\Sa})(q)|^2}-3H_N(p).
\end{align*}
\end{proof}

Next theorem states that the longest timelike geodesic between two spacelike hypersurfaces will perpendicular to both hypersurfaces.
\begin{thm} \label{Thmofperpendicular}
Let $N$ and $\Sa$ be two submanifolds of $L$ such that $\Sa\sbt\ml{I}^+(N)$,
and let $\ga:[0,l]\to L$ be a future directed timelike geodesic such that $\ga(0)\in N, \ga(l)\in\Sa$ and
$\ga$ is the longest curve from $N$ to $\Sa$.
Then $\ga'(0)$ is perpendicular to $N$ and $\ga'(l)$ is perpendicular to $\Sa$.
\end{thm}

\begin{proof}
Suppose that $\ga'(0)$ is not perpendicular to $N$.
Choose $v\in T_{\ga(0)}N$ such that $\lae\ga'(0),v\rae>0$,
and let $\aa(s)$ be a curve in $N$ starting at $\ga(0)$ such that $\aa'(0)=v$.
We construct a variation $\Phi:[0,t]\ts(-\vn,\vn)\to L$ of $\ga(t)$ such that
\begin{align*}
\Phi(t,0)=\ga(t),\quad\Phi(0,s)=\aa(s),\quad\mx{and}\quad\Phi(l,s)=\ga(l).
\end{align*}
Denote $\ga_s(t)=\Phi(t,s)$ the family of timelike curves.
Then the first variation formula (see \cite[Proposition 2, page 264]{OB}) implies
\begin{align*}
\lt.\fc{\mm{d}}{\mm{d}s}\,\mm{Length}(\ga_s)\rt|_{s=0}
&=\fc1l\lt(-\lae\ga'(t),v\rae\Big|_{t=0}^{t=l}+\int_0^l\lae v,\na_{\ga'(t)}\ga'(t)\rae\,\mm{d}t\rt) \\
&=\fc1{l}\lae\ga'(0),v\rae>0.
\end{align*}
Therefore, for small $s$, $\mm{Length}(\ga_s)>\mm{Length}(\ga)$,
which implies $\ga(t)$ is not longest curve from $N$ to $\Sa$, and hence $\ga'(0)$ must be perpendicular to $L$.

Similar argument will get $\ga'(l)$ is perpendicular to $N$.
\end{proof}

\section{Dirichlet problem for SS-CMC equation} \label{section4}
\subsection{Setting the SS-CMC Dirichlet problem} \label{SettingDP}
Let $\Sa:(T=F(X),X,\ta,\phi)$ be an {\small SS-CMC} hypersurface in the Kruskal extension.
In \cite{LL1}, we have computed the {\small SS-CMC} equation:
\begin{align}
F''(X)&+\mbox{e}^{-\frac{r}{2M}}\left(\frac{6M}{r^2}-\frac1r\right)(-F(X)+F'(X)X)(1-(F'(X))^2)\notag\\
&+\frac{12HM\mbox{e}^{-\frac{r}{4M}}}{\sqrt{r}}(1-(F'(X))^2)^{\frac32}=0, \label{CMCequation}
\end{align}
where the spacelike condition is equivalent to $1-(F'(X))^2>0$, and where $H$ is the constant mean curvature.
Remark that $r=r(T,X)=r(F(X),X)$ satisfies the equation (\ref{trans}), namely, $(r-2M)\,\mathrm{e}^{\frac{r}{2M}}=X^2-T^2=X^2-(F(X))^2$;
spherically symmetric condition means that the function $T=F(X)$ is independent of $\ta$ and $\phi$,
and $M$ is the mass of the Schwarzschild spacetime, which is a positive constant.

We can formulate the {\small SS-CMC} Dirichlet problem as follows:

\begin{DP}
Given $H\in\mb{R}$ and boundary data $(T_0,X_0,\ta,\phi)$, $(T_0,-X_0,\ta,\phi)$ in the Kruskal extension,
does there exist a unique hypersurface $\Sa:(T=F(X),X,\ta,\phi)$ satisfying the {\small SS-CMC} equation (\ref{CMCequation}),
the spacelike condition $1-(F'(X))^2>0$, and the boundary value condition $F(X_0)=F(-X_0)=T_0$?
\end{DP}

Since we consider {\small SS-CMC} hypersurfaces in this paper,
the following discussions we only write the $T$-$X$ coordinates $(T,X)$ instead of the full coordinates $(T,X,\theta,\phi)$ in convenience.

\subsection{Existence of the {\small SS-CMC} equation}
\begin{thm} \label{Dirichletexistence}
Dirichlet problem for the {\small SS-CMC} equation with symmetric boundary data {\rm(}$\star${\rm)} is solvable.
\end{thm}

The idea to the proof of the existence is the shooting method.
Take boundary data $(T_0,X_0)$ and $(T_0, -X_0)$ in the middle region for example,
and see Figure~\ref{CMCDirichletExistence2}.
Consider the family of {\small SS-CMC} hypersurfaces $\Sa_{H,c}$ with $c\in(c_H,C_H)$ passing through $(T_0,X_0)$.
The family $\Sa_{H,c}$ is continuously varied with respect to the parameter $c$.
When we observe the position of the ``throat'' $(t_{H,c},r_{H,c})$ of each {\small SS-CMC} hypersurface $\Sa_{H,c}$,
if $c\to c_H$ or $c\to C_H$, then $r_{H,c}$ will tend to $r=r_H$ or $r=R_H$, respectively.
By symmetry and the Intermediate Value Theorem,
we know that there must be some $\Sa_{H,c'}, c'\in(c_H,C_H)$ passing through the other boundary data $(T_0,-X_0)$.

\begin{figure}[h]
\centering
\psfrag{A}{\tiny$(T_0,X_0)$}
\psfrag{B}{\tiny$(T_0,-X_0)$}
\psfrag{r}{$r=r_H$}
\psfrag{R}{$r=R_H$}
\psfrag{T}{$T$}
\psfrag{X}{$X$}
\psfrag{U}{$U$}
\psfrag{V}{$V$}
\includegraphics[height=50mm,width=41mm]{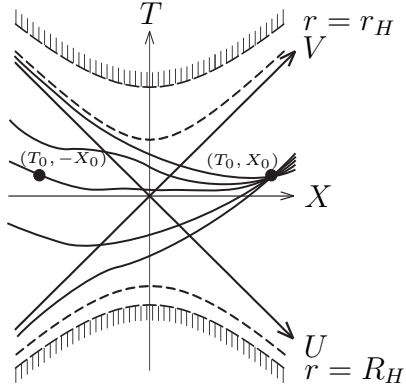}
\caption{Existence of the {\small SS-CMC} Dirichlet problem.} \label{CMCDirichletExistence2}
\end{figure}

\begin{proof}
(A) For $(T_0,X_0)$ and $(T_0,-X_0)$ in the region between $r=r_H$ and $r=R_H$,
denote $(t_0,r_0)$ by the standard Schwarzschild coordinates of $(T_0,X_0)$.
For $c\in(c_H,C_H)$, consider the following curve in the Kruskal extension in null coordinates:
\begin{align*}
\aa(c)=(U(c),V(c))=(U(t_{H,c},r_{H,c}), V(t_{H,c},r_{H,c})),
\end{align*}
where
\begin{align*}
U(c)&=\lt\{
\begin{array}{ll}
-\mm{e}^{-\fc1{4M}(t_0-2r_{H,c}-4M\ln|r_{H,c}-2M|+r_0+2M\ln|r_0-2M|+\int_{r_0}^{r_{H,c}}\bar{f}'(x,c)\,\mm{d}x)} & \mx{in region {\tt I\!I}} \\
+\mm{e}^{-\fc1{4M}(t_0-r_0-2M\ln|r_0-2M|+\int_{r_0}^{r_{H,c}}\tilde{f}'(x,c)\,\mm{d}x)} & \mx{in region {\tt I\!I'}},
\end{array}
\rt.\\
V(c)&=\lt\{
\begin{array}{ll}
+\mm{e}^{\fc1{4M}(t_0+r_0+2M\ln|r_0-2M|+\int_{r_0}^{r_{H,c}}\bar{f}'(x,c)\,\mm{d}x)} & \mx{in region {\tt I\!I}} \\
-\mm{e}^{\fc1{4M}(t_0+2r_{H,c}+4M\ln|r_{H,c}-2M|-r_0-2M\ln|r_0-2M|+\int_{r_0}^{r_{H,c}}\tilde{f}'(x,c)\,\mm{d}x)} & \mx{in region {\tt I\!I'}},
\end{array}
\rt.
\end{align*}
and where $\bar{f}'(x,c)$ and $\tilde{f}'(x,c)$ are (\ref{barfeqn}) and (\ref{tildefeqn}), respectively.
The curve $\aa(c)$ indicates the position of the throat $(t_{H,c},r_{H,c})$ of the {\small SS-CMC} hypersurface $\Sa_{H,c}$
which passes through $(T_0,X_0)$ in the Kruskal extension.
Remark that an {\small SS-CMC} hypersurface is symmetric to its throat in the standard Schwarzschild coordinates.

We will claim that $\fc{V(c')}{U(c')}=-1$ for some $c'\in(c_H,C_H)$, or $\lim\ls_{c\to c'=-8M^3H}\fc{V(c)}{U(c)}=-1$,
which means the throat of $\Sa_{H,c'}$ is located on the $T$-axis,
then the symmetry implies the hypersurface $\Sa_{H,c'}$ must pass through $(T_0,-X_0)$, so we get the existence of the Dirichlet problem ($\star$).

One subtle thing is that functions $\bar{f}'$ and $\tilde{f}'$ have different order behaviors when $c\to-8M^3H$,
so we need to take care of them by the following arguments.

\noindent\ue{If $c<-8M^3H$}, we know that the curve $\aa(c)$ lies in region {\tt I\!I}.
Denote
\begin{align*}
\hat{r}_{H,c}=2M+(2M-r_{H,c})=4M-r_{H,c}\ \mx{ and }\ \hat{t}_{H,c}=t_0+\int_{r_0}^{\hat{r}_{H,c}}f'(x,c)\,\mm{d}x.
\end{align*}
Relation (\ref{trans}) gives
\begin{align*}
\lt.\fc{V(c)}{U(c)}\rt|_{\aa(c)}
&=-\mm{e}^{\fc1{2M}t_{H,c}}
=-\lt(\mm{e}^{\fc1{2M}\hat{t}_{H,c}}\rt)\lt(\mm{e}^{\fc1{2M}(t_{H,c}-\hat{t}_{H,c})}\rt).
\end{align*}
When $c\to -8M^3H^-$, we have
\begin{align*}
\lim_{c\to-8M^3H^-}\mm{e}^{\fc1{2M}\hat{t}_{H,c}}
=\lim_{c\to-8M^3H^-}\mm{e}^{\fc1{2M}\lt(f(r_0)+\int_{r_0}^{\hat{r}_{H,c}}f'(x,c)\,\mm{d}x\rt)}
=\mm{e}^{\fc1{2M}\lt(f(r_0)+\int_{r_0}^{2M}f'(x,-8M^3H)\,\mm{d}x\rt)}.
\end{align*}
The limit exists because the function $f'(x,-8M^3H)$ is of order $O(r-2M)^{-\fc12}$.
Next, we investigate the other limit:
\begin{align*}
\lim_{c\to-8M^3H^-}\mm{e}^{\fc1{2M}(t_{H,c}-\hat{t}_{H,c})}
=\lim_{c\to-8M^3H^-}\mm{e}^{\fc1{2M}\lt(2(2M-r_{H,c})+\int_{r_{H,c}}^{4M-r_{H,c}}\bar{f}'(x,c)\,\mm{d}x\rt)}.
\end{align*}
Since
\begin{align*}
\bar{f}'(x,c)
&=\fc{x^4}{(x-r_{H,c})^{\fc12}P(x,c)},
\end{align*}
where $P(x,c)>0$ is a smooth function,
let $y=x-r_{H,c}$, we get
\begin{align*}
\int_{r_{H,c}}^{4M-r_{H,c}}\bar{f}'(x,c)\,\mm{d}x
&=\int_{0}^{2(2M-r_{H,c})}\fc{(y+r_{H,c})^4}{y^{\fc12}P(y+r_{H,c},c)}\,\mm{d}y
\leq\int_{0}^{2(2M-r_{H,c})}\fc{M}{y^{\fc12}}\,\mm{d}y \\
&=8M(2M-r_{H,c})\to 0 \mx{ as } c\to-8M^3H^-.
\end{align*}
Hence the limit
\begin{align*}
\lim_{c\to-8M^3H^-}\lt.\fc{V(c)}{U(c)}\rt|_{\aa(c)}
=-\mm{e}^{\fc1{2M}\lt(f(r_0)-\int_{2M}^{r_0}f'(x,-8M^3H)\,\mm{d}x\rt)}
\end{align*}
exists. We denote the limit by $L_0$.
On the other hand, from the limit behavior of {\small SS-CMC} hypersurfaces in paper \cite{LL1}, we know
\begin{align*}
\lim_{c\to c_H^+}\lt.\fc{V(c)}{U(c)}\rt|_{\aa(c)}=0.
\end{align*}

\noindent\ue{If $c>-8M^3H$}, we know that the curve lies in region {\tt I\!I'}, and we get
\begin{align*}
\fc{V(c)}{U(c)}
=-\mm{e}^{\fc1{2M}t_{H,c}}
=-\mm{e}^{\fc1{2M}\lt(t_0+r_{H,c}+2M\ln|r_{H,c}-2M|-r_0-2M|r_0-2M|+\int_{r_0}^{r_{H,c}}\tilde{f}'(x,c)\,\mm{d}x\rt)}.
\end{align*}
Similar discussion gives the result that
\begin{align*}
\lim_{c\to-8M^3H^+}\lt.\fc{V(c)}{U(c)}\rt|_{\aa(c)}
=-\mm{e}^{\fc1{2M}\lt(f(r_0)-\int_{2M}^{r_0}f'(x,-8M^3H)\,\mm{d}x\rt)}=L_0.
\end{align*}
Furthermore, from the limit behavior of {\small SS-CMC}  hypersurfaces in paper \cite{LL1}, we know
\begin{align*}
\lim_{c\to C_H^-}\lt.\fc{V(c)}{U(c)}\rt|_{\aa(c)}=-\iy.
\end{align*}
Therefore, by the Intermediate Value Theorem, we get
\begin{itemize}
\item[(1)] If $L_0<-1$, there exists $c'\in(c_H,-8M^3H)$ such that $\fc{V(c')}{U(c')}=-1$.
\item[(2)] If $L_0=-1$, then $c'=-8M^3H$ satisfies $\lim\ls_{c\to c'}\fc{V(c)}{U(c)}=-1$.
\item[(3)] If $L_0>-1$, there exists $c'\in(-8M^3H,C_H)$ such that $\fc{V(c')}{U(c')}=-1$.
\end{itemize}

\noindent (B) For $(T_0,X_0)$ and $(T_0,-X_0)$ in the region between $r=r_H$ and $r=0$,
denote $(t_0,r_0)$ by the standard Schwarzschild coordinates of $(T_0,X_0)$.
Let $c_0=-Hr_0^3-r_0^{\fc32}(2M-r_0)^{\fc12}.$
For $c\in(c_H,c_0)$,
the curve
\begin{align*}
\aa(c)
&=(U(c),V(c)) \\
&=\lt(
-\mm{e}^{-\fc1{4M}\lt(\int_{r_0}^{r_{H,c}}f'(x,c)\,\mm{d}x-r_0-2M\ln|r_0-2M|\rt)},
\mm{e}^{\fc1{4M}\lt(\int_{r_0}^{r_{H,c}}f'(x,c)\,\mm{d}x+r_0+2M\ln|r_0-2M|\rt)}\rt).
\end{align*}
will trace the position of the maximum radius of the Schwarzschild coordinates sphere of the {\small SS-CMC} hypersurface
$\Sa_{H,c}$ which passes through $(T_0,X_0)$.
Since $\lim\ls_{c\to c_0^-}\fc{V(c)}{U(c)}=-\mm{e}^{\fc1{2M}t_{H,c}}<-1$ and $\lim\ls_{c\to c_H^+}\fc{V(c)}{U(c)}=0$,
by the Intermediate Value Theorem, there exists $c'\in(c_H,c_0)$ such that $\fc{V(c')}{U(c')}=-1$.

\noindent (C) For $(T_0,X_0)$ and $(T_0,-X_0)$ in the region between $r=R_H$ and $r=0$,
denote $(t_0,r_0)$ by the standard Schwarzschild coordinates of $(T_0,X_0)$.
Let $c_0=-Hr_0^3+r_0^{\fc32}(2M-r_0)^{\fc12}$.
For $c\in(c_0,C_H)$,
the curve
\begin{align*}
\aa(c)&=(U(c),V(c))\\
&=\lt(\mm{e}^{-\fc1{4M}\lt(\int_{r_0}^{r_{H,c}}f'(x,c)\,\mm{d}x-r_0-2M\ln|r_0-2M|\rt)},
-\mm{e}^{\fc1{4M}\lt(\int_{r_0}^{r_{H,c}}f'(x,c)\,\mm{d}x+r_0+2M\ln|r_0-2M|\rt)}\rt).
\end{align*}
will trace the position of the maximum radius of the Schwarzschild coordinates sphere in the {\small SS-CMC} hypersurface $\Sa_{H,c}$
which passes through $(T_0,X_0)$.
Since $\lim\ls_{c\to c_0^+}\fc{V(c)}{U(c)}=-\mm{e}^{\fc1{2M}t_{H,c}}>-1$ and $\lim\ls_{c\to C_H^-}\fc{V(c)}{U(c)}=-\iy$,
by the Intermediate Value Theorem, there exists $c'\in(c_0,C_H)$ such that $\fc{V(c')}{U(c')}=-1$.
\end{proof}

It is easy to find that solutions of the {\small SS-CMC} equation with symmetric boundary data are symmetric about the $T$-axis in the Kruskal extension.
\begin{defn}
We say an {\small SS-CMC} hypersurface $\Sa:(T=F(X),X,\ta,\phi)$ in the Kruskal extension is {\it $T$-axisymmetric} if $F(-X)=F(X)$ for all $X$.
\end{defn}

\begin{thm} \label{Taxisymmetric}
All {\small SS-CMC} hypersurfaces satisfying the Dirichlet problem {\rm(}$\star${\rm)} are $T$-axisymmetric.
That is, $F(-X)=F(X)$ for all $X$.
\end{thm}

\begin{proof}
Suppose that an {\small SS-CMC} hypersurface $\Sa$ satisfying {\rm(}$\star${\rm)} is not $T$-axisymmetric.
We consider its $T$-axisymmetric reflection hypersurface, called $\tilde{\Sa}$.
Then $\Sa$ and $\tilde{\Sa}$ have different parameter $c$.
On the other hand, the reflection does not change the radius of the throat $r=r_{H,c}$
or the maximum radius of the Schwarzschild coordinates sphere, so $\Sa$ and $\tilde{\Sa}$ must have the same value $c$,
and it leads to the contradiction.
\end{proof}

\subsection{Uniqueness of the {\small SS-CMC} equation}
We will prove the uniqueness of the {\small SS-CMC} Dirichlet problem in this subsection.
Before that, we need to find more properties about the {\small SS-CMC} hypersurfaces.
Theorem~\ref{Dirichletexistence} and \ref{Taxisymmetric} imply that for fixed boundary data $(T_0, X_0)$ and $(T_0,-X_0)$, and for every $H\in\mb{R}$,
there exists a {\small TSS-CMC} hypersurface $\Sa_{H,c}$ satisfying the Dirichlet problem ($\star$).
Next theorem shows that these hypersurfaces $\Sa_{H,c}$ are continuously varied with respect to the mean curvature $H$.

\begin{thm} \label{Continuity}
For every $H\in\mb{R}$, denote $\Sa_{H,c(H)}:(T=F_H(X),X,\theta,\phi)$
an {\small SS-CMC} hypersurface satisfying the {\small SS-CMC} equation {\rm(\ref{CMCequation})},
the spacelike condition $1-(F_H'(X))^2>0$, and the boundary value condition $F_H(X_0)=F_H(-X_0)=T_0$.
Then the family $\Sa_{H,c(H)}$ {\rm(}the function $F_H(X)${\rm)} is continuously varied with respect to the mean curvature $H$.
\end{thm}
\begin{proof}
(A) For $(T_0,X_0)$ and $(T_0,-X_0)$ between $r=r_H$ and $r=R_H$, consider two functions
\begin{align*}
\bar{G}(r,H)=t_0-r-2M\ln|r-2M|+r_0+2M\ln|r_0-2M|-\int_{r}^{r_0}\bar{f}'(x,H,r)\,\mm{d}x,
\end{align*}
where
\begin{align*}
\bar{f}'(x,H,r)&=\fc{x^4}{(Hx^3+c)^2+x^3(x-2M)-(Hx^3+c)\st{(Hx^3+c)^2+x^3(x-2M)}}, \\
c&=c(H,r)=-Hr^3-r^{\fc32}(2M-r)^{\fc12},
\end{align*}
and
\begin{align*}
\tilde{G}(r,H)=t_0+r+2M\ln|r-2M|-r_0-2M\ln|r_0-2M|-\int_{r}^{r_0}\tilde{f}'(x,H,r)\,\mm{d}x,
\end{align*}
where
\begin{align*}
\tilde{f}'(x,H,r)&=\fc{-x^4}{(Hx^3+c)^2+x^3(x-2M)+(Hx^3+c)\st{(Hx^3+c)^2+x^3(x-2M)}}, \\
c&=c(H,r)=-Hr^3+r^{\fc32}(2M-r)^{\fc12}.
\end{align*}
Both domain of $\bar{G}$ and $\tilde{G}$ are a subset of $(0,2M]\ts\mb{R}$.
Since we have the relation
\begin{align*}
\fc{V(r,H)}{U(r,H)}=-\mm{e}^{-\fc1{2M}\bar{G}(r,H)}\ \ (\mx{region } {\tt I\!I}) \quad\mx{and}
\quad\fc{V(r,H)}{U(r,H)}=-\mm{e}^{-\fc1{2M}\tilde{G}(r,H)}\ \ (\mx{region } {\tt I\!I'}),
\end{align*}
we use the solutions of equations $\bar{G}(r,H)=0$ or $\tilde{G}(r,H)=0$
to label the throat of {\small TSS-CMC} hypersurfaces satisfying the Dirichlet problem ($\star$).
More precisely,
Theorem~\ref{Dirichletexistence} shows that for every $H\in\mb{R}$, there exists $r=r(H)$ such that $\bar{G}(r,H)=0$ or $\tilde{G}(r,H)=0$.
Hence the set $(r,H)$ of $\bar{G}(r,H)=0$ or $\tilde{G}(r,H)=0$ correspond to {\small TSS-CMC} hypersurfaces $\Sa_{H,c(H)}$.

We compute the partial derivative of the function $\bar{G}(r,H)$ with respect to $H$:
\begin{align*}
\fc{\pl \bar{G}}{\pl H}
&=-\int_{r}^{r_0}\fc{\pl \bar{f}'}{\pl H}(x,H,r)\,\mm{d}x
=-\int_{r}^{r_0}\fc{\pl f'}{\pl H}(x,H,r)\,\mm{d}x \\
&=-\int_r^{r_0}\fc1{h(x)}\fc{\fc{\pl l}{\pl H}(x,H,r)}{(1+l^2(x,H,r))^{\fc32}}\,\mm{d}x<0.
\end{align*}
By the Implicit Function Theorem,
for every $(r,H)$ with $\bar{G}(r,H)=0$,
there exists an interval $(r-\da,r+\da)$ such that the solution of $\bar{G}(r,H)=0$
in the interval can be written as a graph of a function $H=H(r)$.

Similarly, we compute
\begin{align*}
\fc{\pl \tilde{G}}{\pl H}
&=-\int_{r}^{r_0}\fc{\pl \tilde{f}'}{\pl H}(x,H,r)\,\mm{d}x
=-\int_{r}^{r_0}\fc{\pl f'}{\pl H}(x,H,r)\,\mm{d}x  \\
&=-\int_r^{r_0}\fc1{h(x)}\fc{\fc{\pl l}{\pl H}(x,H,r)}{(1+l^2(x,H,r))^{\fc32}}\,\mm{d}x<0.
\end{align*}

By the Implicit Function Theorem,
for every $(r,H)$ with $\tilde{G}(r,H)=0$,
there exists an interval $(r-\da,r+\da)$ such that the solution of $\tilde{G}(r,H)=0$
in the interval can be written as a graph of a function $H=H(r)$.

\noindent(B) For $(T_0,X_0)$ and $(T_0,-X_0)$ between $r=r_H$ and $r=0$ (between $r=0$ and $r=R_H$), consider the function
\begin{align*}
G(r,H)=t_0+\int_{r_0}^rf'(x,c)\,\mm{d}x\quad\mx{with}\quad\fc{V(r,H)}{U(r,H)}=-\mm{e}^{-\fc1{2M}G(r,H)}
\end{align*}
in region {\tt I\!I} (in region {\tt I\!I'}).
Since $\fc{\pl G}{\pl H}>0$ in region {\tt I\!I} ($\fc{\pl G}{\pl H}>0$ in region {\tt I\!I'}),
by the Implicit Function Theorem, for every $(r,H)$ with $G(r,H)=0$,
there exists an interval $(r-\da,r+\da)$ such that the solution of $G(r,H)=0$
in the interval can be written as a graph of a function $H=H(r)$.
\end{proof}

The result of Theorem~\ref{Continuity} is illustrated in Figure~\ref{GraphofHvariedI}.
We plot the graph of $H=H(r)$ in the $r$-$H$ plane.

\begin{figure}[h]
\centering
\psfrag{A}{Region {\tt I\!I}}
\psfrag{B}{Region {\tt I\!I'}}
\psfrag{S}{$P$}
\psfrag{a}{(a)}
\psfrag{b}{(b)}
\psfrag{c}{(c)}
\psfrag{P}{\footnotesize$\aa_1(r)$}
\psfrag{Q}{\footnotesize$\ba_1(r)$}
\psfrag{R}{\footnotesize$\ga_1(r)$}
\psfrag{p}{\footnotesize$\aa_2(r)$}
\psfrag{q}{\footnotesize$\ba_2(r)$}
\psfrag{s}{\footnotesize$\ga_2(r)$}
\psfrag{X}{\footnotesize$\ba_1\cup\ba_2(T)$}
\psfrag{Y}{\footnotesize$\ga_1(T)$}
\psfrag{Z}{\footnotesize$\ga_2(T)$}
\psfrag{H}{\footnotesize$H$}
\psfrag{r}{\footnotesize$r$}
\psfrag{T}{\footnotesize$T$}
\psfrag{M}{\footnotesize$2M$}
\psfrag{m}{\footnotesize$\st{2M}$}
\psfrag{n}{\footnotesize$-\st{2M}$}
\includegraphics[height=45mm,width=140mm]{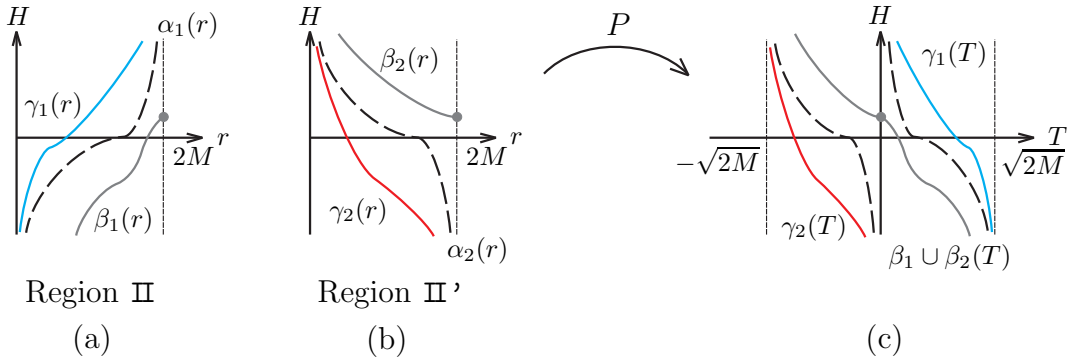}
\caption{(a) and (b): Curves $\ba_1(r), \ba_2(r)$, and $\ga_1\cup\ga_2(r)$ are solutions of $\bar{G}(r,H)=0, \tilde{G}(r,H)=0$, and $G(r,H)=0$, respectively.
They correspond to {\small TSS-CMC} hypersurfaces $\Sa_{H,c(H)}$ for different mean curvature $H$.
(c) We map all curves in Figure~\ref{GraphofHvariedI} (a) and (b) to the $T$-$H$ plane. } \label{GraphofHvariedI}
\end{figure}
In Figure~\ref{GraphofHvariedI} (a), the dotted curve $\aa_1(r)$ is the graph of $H(r)=\fc{2r-3M}{3\st{r^3(2M-r)}}$.
Each point on $\aa_1(r)$ represents a cylindrical hypersurface with constant mean curvature $H$ in region {\tt I\!I}.
Curves $\ba_1(r)$ and $\ga_1(r)$ are solutions of $\bar{G}(r,H)=0$ and $G(r,H)=0$, respectively.
Similarly, in Figure~\ref{GraphofHvariedI} (b), the dotted curve $\aa_2(r)$ is the graph of $H(r)=\fc{3M-2r}{3\st{r^3(2M-r)}}$.
Curves $\ba_2(r)$ and $\ga_2(r)$ are solutions of $\tilde{G}(r,H)=0$ and $G(r,H)=0$, respectively.
Theorem~\ref{Continuity} implies that every curve $\ba_1, \ba_2, \ga_1$, and $\ga_2$ is the graph of some continuous function $H=H(r)$.
Theorem~\ref{Dirichletexistence} implies that $\ba_1(2M)=\ba_2(2M)$.

Consider a map $P$ from these two $r$-$H$ planes in Figure~\ref{GraphofHvariedI} (a) and (b) to the $T$-$H$ plane in Figure~\ref{GraphofHvariedI} (c) by
\begin{align*}
H\stackrel{P}{\longmapsto} H\quad\mx{and}\quad r\stackrel{P}{\longmapsto} T=
\lt\{
\begin{array}{ll}
\st{2M-r}\,\mm{e}^{\fc{r}{4M}} & \mx{if } r\in(0,2M] \mx{ in region } {\tt I\!I} \\
-\st{2M-r}\,\mm{e}^{\fc{r}{4M}} & \mx{if } r\in(0,2M] \mx{ in region } {\tt I\!I'}.
\end{array}
\rt.
\end{align*}
Theorem~\ref{Continuity} shows that each curve $\ba_1\cup\ba_2(T)$, $\ga_1(T)$, and $\ga_2(T)$
is the graph of some continuous function $H=H(T)$ in the $T$-$H$ plane.

In order to prove the uniqueness of the {\small SS-CMC} Dirichlet problem $(\star)$,
it is equivalent to prove that each curve $\ba_1\cup\ba_2(T)$, $\ga_1(T)$ and $\ga_2(T)$
is the graph of a ``decreasing'' function in the $T$-$H$ plane.

\begin{thm} \label{uniqueness}
The solution of the Dirichlet problem for {\small SS-CMC} equation with symmetric boundary data {\rm(}$\star${\rm)} is unique.
\end{thm}

\begin{proof}
(A) Given $H_0\in\mb{R}$, suppose that there are two {\small TSS-CMC} hypersurfaces $\Sa_1:(T=F_1(X),X)$ and
$\Sa_2:(T=F_2(X),X)$ with constant mean curvature $H_0$ satisfying the Dirichlet problem ($\star$),
and suppose that $F_1(X)<F_2(X)$ in $(-X_0,X_0)$.
Then the curve $\ba_1\cup\ba_2(T)$ passes through $(T_1=F_1(X=0),H_0)$ and $(T_2=F_2(X=0),H_0)$ in the $T$-$H$ plane,
and $\ba_1\cup\ba_2(T)$ is no longer the graph of a decreasing function in the $T$-$H$ plane,
which implies that there must be some increasing part on some interval $I\sbt[T_1,T_2]$.
See Figure~\ref{UniquenessofDirichletProblem}.
Then we can take two {\small TSS-CMC} hypersurfaces $N$ and $\Sa$ such that $\Sa\sbt\ml{I}^+(N)$ in the region $T\ts[-X_0,X_0]$
in the Kruskal extension and $H_{N}<H_{\Sa}$.

\begin{figure}[h]
\centering
\psfrag{C}{\footnotesize$\ba_1\cup\ba_2(T)$}
\psfrag{Y}{\footnotesize$\ga_1(T)$}
\psfrag{Z}{\footnotesize$\ga_2(T)$}
\psfrag{H}{$H$}
\psfrag{T}{$T$}
\psfrag{X}{$X$}
\psfrag{A}{\tiny$\Sa_1$}
\psfrag{B}{\tiny$\Sa_2$}
\psfrag{I}{\footnotesize$I$}
\psfrag{h}{\footnotesize$H_0$}
\psfrag{t}{\tiny$T_1$}
\psfrag{p}{\tiny$T_2$}
\psfrag{r}{\footnotesize$r=r_H$}
\psfrag{R}{\footnotesize$r=R_H$}
\psfrag{N}{\tiny$N$}
\psfrag{S}{\tiny$\Sa$}
\psfrag{P}{\tiny$X_0$}
\psfrag{Q}{\tiny$-X_0$}
\includegraphics[height=37mm,width=125mm]{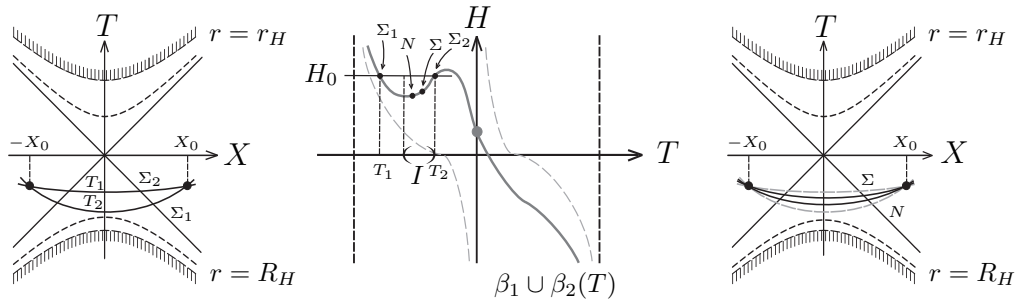}
\caption{Choose two hypersurfaces $N$ and $\Sa$ in the increasing part of $\ba_1\cup\ba_2(T)$,
and then find a contradiction to the nonuniqueness of the {\small SS-CMC} Dirichlet problem.} \label{UniquenessofDirichletProblem}
\end{figure}
Consider $d_{N}|_{\Sa}$ the Lorentzian distance function restricted on the spacelike hypersurface $\Sa$.
Since $N\neq\Sa$, we know that the maximum value of $d_N|_{\Sa}$
is positive and will achieve at some point $q$, which is in the interior of $\Sa\cap T\ts(-X_0,X_0)$ in the Kruskal extension.
We apply the estimate (\ref{EstimateofLaplace}) to $N$ and $\Sa$ and get
\begin{align*}
0\geq\Da_{\Sa}(d_{N}|_{\Sa})(q)
&\geq\overline{\mm{Hess}}(d_{N}(q))(\nu,\nu)+3H_{\Sa}(q)\st{1+|\na(d_{N}|_{\Sa})(q)|^2}-3H_{N}(p) \\
&\geq 0+3H_{\Sa}(q)-3H_{N}(p)>0,
\end{align*}
which is a contradiction. Remark that
$\overline{\mm{Hess}}(d_{N}(q))(\nu,\nu)=0$ because of the
perpendicular property in Theorem~\ref{Thmofperpendicular}.
\end{proof}

\section{Applications to CMC foliation conjecture}
Malec and \'{O} Murchadha in \cite{MO} constructed a family of $T$-axisymmetric {\small SS-CMC} ({\small TSS-CMC}) hypersurfaces in the Kruskal extension.
Each {\small TSS-CMC} hypersurface in this family shares the same constant mean curvature $H$.
They conjectured this {\small TSS-CMC} family will foliate the Kruskal extension in \cite{MO} without rigorous mathematical proof.

From the viewpoint of the {\small SS-CMC} Dirichlet problem, we will prove the existence of the {\small TSS-CMC} foliation in the Kruskal extension.
Given $H\in\mb{R}$, for all symmetric pairs $(T_0,X_0)$ and $(T_0,-X_0)$ in the Kruskal extension,
we collect all {\small TSS-CMC} hypersurfaces with mean curvature $H$ and pass through $(T_0,X_0)$ and $(T_0,-X_0)$.
We denote this {\small TSS-CMC} family by $\{\Sa_H\}$.
Since $(T_0,X_0)$ and $(T_0,-X_0)$ will be taken in the whole Kruskal extension,
$\{\Sa_H\}$ must cover the whole Kruskal extension.
If any two hypersurfaces $\Sa_1, \Sa_2\in\{\Sa_H\}$ are not disjoint,
then they must intersect at some symmetric pairs.
However, the uniqueness of the {\small SS-CMC} Dirichlet problem implies $\Sa_1=\Sa_2$.
Therefore, we can conclude the following theorem:
\begin{thm}
The existence and uniqueness of the Dirichlet problem {\rm(}$\star${\rm)} is equivalent to the existence of {\small TSS-CMC} foliation in the Kruskal extension.
\end{thm}

\fontsize{11}{14pt plus.5pt minus.4pt}\selectfont

\end{document}